\documentclass[12pt]{amsart}
\usepackage[margin=30mm]{geometry}

\title[Jacobi ensembles]{Edge scaling of the $\beta$-Jacobi ensemble}
\author{Diane Holcomb, Gregorio R. Moreno Flores}
\address{University of Wisconsin-Madison, Department of Mathematics, 480 Lincoln Dr., Madison WI 53706}
\email{holcomb@math.wisc.edu, moreno@math.wisc.edu}

\usepackage{amsfonts}
\usepackage{graphicx}
\usepackage{amsmath}
\usepackage{amsthm}
\usepackage{amssymb}

    \newtheorem{theorem}{Theorem}
    \newtheorem{lemma}[theorem]{Lemma}
    \newtheorem{proposition}[theorem]{Proposition}
    
    \newtheorem{claim}[theorem]{Claim}

\theoremstyle{definition} 
    
    \newtheorem{remark}[theorem]{Remark}

\newcommand{\lstar}{{\raise-0.15ex\hbox{$\scriptstyle \ast$}}}

\newcommand{\var}{\text{Var }}
\theoremstyle{remark} 

\newcommand{\Balpha}{\underline{\alpha}}
\newcommand{\Btheta}{\underline{\theta}}
\newcommand{\Blambda}{\underline{\lambda}}
\newcommand{\Bq}{\underline{q}}
\newcommand{\Bx}{\underline{x}}
\newcommand{\By}{\underline{y}}
\newcommand{\Ba}{\underline{a}}
\newcommand{\Bb}{\underline{b}}

\newcommand{\rr}{\mathbb{R}}
\newcommand{\ind}{{\bf 1}}

\begin{document}

\begin{abstract}
 We study the scaling limit of the spectrum of the $\beta$-Jacobi ensemble at the soft-edge and hard-edge for general values of $\beta$. We show that the limiting point processes correspond respectively to the stochastic Airy and Bessel point processes introduced in \cite{RRV} and \cite{RR}.
\end{abstract}

\maketitle

\section{Introduction}

Random matrices are well known in the mathematical community as a model coming from nuclear physics, although they were first introduced in statistics. The model considered by Wishart \cite{Wish} (now known as {\it Wishart matrices}) can be described as follows: let $\mathcal{M}_{n \times m}$ denote the space of $n \times m$ matrices with complex (resp. real) valued entries. Take a matrix $X \in \mathcal{M}_{n\times m}$ with independent complex (resp. real) Gaussian entries, then $XX^T \in \mathcal{M}_{n\times n}$  is said to have a $(n,m)$-Wishart distribution. It corresponds to the sample correlation matrix for a sample drawn from a multivariate complex (resp. real) normal distribution.

However the most famous model of random matrices is probably the GUE ensemble: let $\{ \xi_{l,j},\, \eta_{l,j}:\, 1\leq l\leq j\leq N \}$ be an array of independent standard normal random variables. The diagonal entries of a GUE matrix correspond to $X_{l,l}=\xi_{l,l}$ while the off-diagonal entries are defined as $X_{l,j}=2^{-\frac12}(\xi_{l,j}+i \eta_{l,j})$ and $X_{l,i}=2^{-\frac12}(\xi_{l,j}-i \eta_{l,j})$, for $l<j$.  This model was introduced by Wigner as a toy model for the spectrum of heavy atoms, by observing that the spacings between consecutive eigenvalues of a GUE random matrix mimic the spacing between different energy levels of these atoms (see \cite{M} and references therin).

In both models, when the size of the matrix grows, the properly rescaled empirical spectral measure converges to a compactly supported distribution called the semi-circular law in the case of the GUE and the Marchenko-Pastur law in the case of the Wishart model. In this last case, the limit law depends on the ratio of the parameters $m$ and $n$.

The compact support of the limiting empirical laws suggests that an interesting phenomenon should occur near to the edge of the spectrum. This was first studied in the GUE case where it was discovered that the law of the bottom of the spectrum, after centering and rescaling, converges to a point process (the stochastic Airy process). In particular, the law of the properly centered and scaled smallest eigenvalue converges the GUE Tracy-Widom distribution \cite{TW1}. 

Substantial differences appear between the GUE and the Wishart ensemble. In the later case, for certain values of the parameters, the lower edge of the spectrum will converge to $0$ which is the leftmost possible value for the eigenvalues of a positive definite matrix. This in turn suggests that the spectrum should converge to a different limiting point process supported on the positive half-line, the Bessel process \cite{TW2}. This extreme situation is known as a hard-edge, in opposition to the soft-edge situation described in the previous paragraph. While the Wishart ensemble can lead to a soft-edge or a hard-edge depending on the parameters, the GUE only leads to a soft-edge. One can, of course, also consider the upper edge of the spectrum which gives a soft-edge in both cases.

A straightforward generalization of the GUE and the Wishart ensemble leads to the $\beta$-Hermite and $\beta$-Laguerre ensembles. The joint density of these point processes are given respectively by: 
\begin{eqnarray*}
 f^{Hermite}&=& \frac{1}{Z_{\beta,n}} \prod_{i=1}^n e^{- \frac{\beta}{2} (\lambda_i^2/2)} \prod_{j<k} |\lambda_j-\lambda_k|^\beta\\
 f^{Laguerre}&=& \frac{1}{Z_{\beta,n,m}}\prod_{i=1}^n \lambda_i^{\frac{\beta}{2}(m-n+1)-1} e^{-\frac{\beta}{2} \lambda_i} \prod_{j<k} |\lambda_j-\lambda_k|^\beta
\end{eqnarray*}
The $\beta$-Hermite ensembles include the classical random matrix ensembles, such as the GUE (for $\beta=2$), the GOE (for $\beta=1$), which is the real analogue of the GUE, and the GSE (for $\beta=4$), the symplectic analogue of the GUE. The $\beta$-Laguerre ensembles correspond to the complex Wishart model when $\beta=2$. $\beta=1$ and $4$ correspond respectively to real and quaternion valued analogues.

In both settings, the cases $\beta=1,\, 2$ and $4$ have the particularity to be solvable and can be studied by means of asymptotics of orthogonal polynomials \cite{M}. In the general $\beta$ case, the solvability is lost, together with a natural interpretation in term of classical random matrix ensembles. The soft and hard edge limits for these general ensembles were considered in \cite{RRV} and \cite{RR} respectively, leading to the stochastic Airy process and the stochastic Bessel process, respectively, in the soft-edge and the hard-edge case. This approach was anticipated in \cite{ES}, where tridiagonal matrix models associated to the $\beta$-Laguerre and Hermite ensemble \cite{DE} were conjecture to converge to continuum random operators. It is worth noting that the link between these random operators and the classical results involving the Tracy-Widom distribution remains obscur. To the best of our knowledge, the only result in this direction is \cite{BV} where a spiked-random matrix model is investigated.

\noindent The scaling limit of eigenvalues in the bulk of the spectrum for the $\beta$-ensembles is studied in \cite{VV} and \cite{JV}. We will not address this type of questions here.

\vspace{3ex}  

The object of study in this paper is an ``eigenvalue'' ensemble derived from the following class of matrices.  We take an $n_1$ by $n$ matrix $M$ with each entry independently drawn from the standard normal distribution, the matrix $X=M^TM$ is one of the $n$ by $n$ Hermitian matrix models introduced by Wishart.  If we assume $n_1\geq n$ we know the resulting matrix is almost surely invertible and so it is reasonable to consider $A=X^{1/2}(X+Y)^{-1}X^{1/2}$ where $Y$ is constructed in the same manner with parameters $n_2, n$.  Consideration of this type of matrix first arose in statistics (MANOVA, or multivariate analysis of variance, uses the base matrix model to study the interdependence of several dependent and independent variables). The resulting matrix $A$ is again Hermitian and joint density function of its eigenvalues is given by
$$f_{\beta,n,n_1,n_2}(\lambda_1,\dots,\lambda_n)=\frac{1}{Z_{\beta,n,n_1,n_2}}\prod_{i=1}^n \lambda_i^a (1-\lambda_i)^b \prod_{j<k} |\lambda_j-\lambda_k|^\beta\, {\bf 1}_{\{\lambda_i \in [0,1],\, \forall \, i\}}$$
where $a=\frac\beta 2 (n_1-n+1)-1, b=\frac \beta 2 (n_2-n+1)-1$, $Z$ is a constant dependent on the parameters, $n_1, n_2\ge n$ and $\beta=1$.
If instead of drawing from the standard normal distribution we had chosen a complex normal distribution, we would have had the same joint density function with $\beta=2$. These cases were extensively studied by Johnstone in \cite{J}.

Note that the eigenvalues lie inside the interval $[0,1]$, unlike the Hermite or Laguerre cases where the spectrum is unbounded. The limiting spectral density is given by the following theorem:

\begin{theorem}\cite{CC} Let $\mu_{n}= \frac{1}{n} \sum_i \delta_{\lambda_i}$ with $\{\lambda_i\}$ the eigenvalues of $J(n,n_1,n_2,\beta)$, $n_1,n_2>n$, $n_1/n\to \gamma_1$, and 
 $n_2/n\to \gamma_2$, then $\mu_n(x)$ converges weakly to $\rho(x)$, where
\[
  \rho(x) = \frac{2\pi}{\gamma_1+\gamma_2} \cdot  \frac{\sqrt{(\Lambda_+ - x)(x- \Lambda_-)}}{x(1-x)}.
\]
Here  $\Lambda_{\pm}$ denotes the upper and lower edges of the spectrum which are given by
\begin{eqnarray}\label{formula-edges}
\Lambda_\pm= \left( \frac{ \sqrt{\gamma_1(\gamma_1+\gamma_2-1)}}{\gamma_1+\gamma_2}\pm \frac{\sqrt{\gamma_2 }}{\gamma_1+\gamma_2}\right)^2.
\end{eqnarray}

\end{theorem}

Note that the expected edges of the $J(n,n_1,n_2,\beta)$ spectrum are given by
\[ \Lambda_\pm = \left( \frac{ \sqrt{n_1(n_1+n_2-n)}}{n_1+n_2} \pm \frac{\sqrt{ nn_2}}{n_1+n_2} \right)^2.\]

In general we need not restrict ourselves to the cases where $\beta=1,2$ or $4$.  For general $\beta$ the joint density function defines a more general random point process called the $\beta$-Jacobi ensemble and will be denoted $J(n,n_1,n_2,\beta)$. This is a point process that can no longer be associated to any natural ensemble of random matrices for $\beta \neq 1,\, 2$ or $4$. However, we will sometimes refer to the points of the point process as the ``eigenvalues" of the $\beta$-Jacobi ensemble. This slight abuse of terminology will in fact be made rigourous later as we will see that it is possible to construct a family of tridiagonal matrices in such a way that the law of their eigenvalues corresponds to a $\beta$-Jacobi ensemble.

\vspace{2ex}

This paper will focus on the behavior at the edge of the spectrum as $n,\, n_1$ and $n_2$ grow to infinity. This approach relies on a tridiagonal representation of the $\beta$-Jacobi ensembles \cite{KN} and the techniques developped in \cite{RRV} and \cite{RR}. The higher number of parameters makes the phase diagram richer. As there are both a lower and an upper bound on the spectrum, appropriate tuning of the parameters can lead to any combination of soft/hard upper/lower edges. With our notations, the asymptotics of $n_1$ with respect to $n$ will determine the nature of the lower edge, while respective asymptotics for $n_2$ will determine the upper edge.

More precisely, if we write $n_1=n+a_n$ then if $a_n\to a<\infty$ there will be a hard edge at the origin (it is easy to see from (\ref{formula-edges}) that the expected lower edge converges to $0$). In this case, we will prove that the rescaled spectrum converges to the stochastic Bessel process from \cite{ES}. Our analysis is based on general results proved in \cite{RR}.

If $a_n\to \infty$, we will see a soft edge. This situation is more delicate, as it is possible to have $a_n \to +\infty$ and still have the lower edge converging to $0$. We will restrict to the situation $\liminf n_1/n > 1$, which implies that the lower edge will stay away from $0$. We will prove that the rescaled and center eigenvalues converge to the stochastic Airy operator introduced in \cite{ES}. Our approach follows \cite{RRV}. The most general soft-edge case, even in the Laguerre case, remains open (see \cite{RRV}).  Similar considerations apply for the upper edge replacing $n_1$ by $n_2$.

\vspace{3ex}

Several works have been devoted to the edge behavior of the $\beta$-Jacobi ensembles with varying degree of generality (see for example \cite{D,DK, DN, NF, NW}). The works \cite{C} and \cite{J} are restricted to the cases $\beta=1$ and $2$. Johnstone \cite{J} also provides fluctuation results for the top eigenvalue in the cases $\beta=1$ and $2$. All the aformentioned works treat the case of a single extreme eigenvalue. We note that the work \cite{J} uses the results from \cite{RRV} and \cite{RR} to study very degenerate asymptotics for $n_1$ and $n_2$ for which the Jacobi ensemble approximates the Laguerre ensemble. 

 The bulk behavior of a larger class of $\beta$-ensembles that includes the $\beta$-Jacobi ensemble (at least for a whide family of parameters) as been studied in \cite{BEY}.

\vspace{3ex}

This work is organised as follows: in Section \ref{intro-tri} we recall the tridiagonal representation for the $\beta$-Jacobi ensemble proved in \cite{KN} (an alternative approach will be described in the appendix). Section \ref{intro-soft} and \ref{intro-hard} present our results on the soft-edge and the hard-edge respectively. Proofs in the soft edge cases are presented in Section \ref{se} while the hard-edge case is treated in Section \ref{he}. Many details are similar to the corresponding proofs in \cite{RRV} and \cite{RR} and will be omited.

\subsection*{Acknowledgements} The authors would like to thank Benedek Valk\'o for introducing them to the field of random matrices and suggesting this problem.

\section{Results}

\subsection{Tridiagonal representation}\label{intro-tri}

The core of our approach, following \cite{RRV} and \cite{RR}, rests heavily on a tridiagonal representation of $J_{n,n_1,n_2,\beta}$.  The first tridiagonal representation for the $\beta$-Jacobi ensemble was given by Killip and Nenciu \cite{KN}.  We will work with a somewhat simplified model which was introduced by Sutton (\cite{S}, Chapter 5).

\begin{theorem}
\label{tridiagonal}\cite{KN}
The joint density of the eigenvalues of $MM^T$ is given by $f_{\beta,n, n_1, n_2}(\Blambda)$, where
\[M=M_{n,n_1,n_2,\beta}= \left[ \begin{array}{ccccc} C_1 \tilde S_1 & &&& \\
S_2\tilde C_1& C_2 \tilde S_2 &&&\\
& S_3 \tilde C_2 & \ddots && \\
&& \ddots && \\
&&& S_n \tilde C_{n-1} & C_n \end{array}\right]\]
with $C_i^2+S_i^2=1, \tilde C_i^2+\tilde S_i^2=1$ and
\begin{eqnarray*}
C_k&\sim& \sqrt{\textup{Beta}\left(\frac{\beta}{2}(n_1-n+k),\frac{\beta}{2}(n_2-n+k)\right)}\\
\tilde C_k&\sim& \sqrt{\textup{Beta}\left(\frac{\beta}{2}k,\frac{\beta}{2}(n_1+n_2-2n+k+1)\right)}.
\end{eqnarray*}
\end{theorem} 

Note that $MM^T$ is indeed tridiagonal.  A proof of this theorem is included in the appendix for the sake of completeness.



\subsection{Soft edge limit}\label{intro-soft}

The proof of the soft edge limit is derived using a more general limiting result by Ram\'irez, Rider and Vir\'ag \cite{RRV}.  This result embeds a sequence of tridiagonal matrices as operators on $L^2[0,\infty)$ and gives conditions for a weak limit under which convergence of the eigenvalues also holds.  We use this result to show that the point process $J(n,n_1,n_2,\beta)$ converges to eigenvalues of a random operator.


We begin by defining the ``stochastic Airy operator" (SAE$_\beta$) .  Let 
\begin{equation}
\mathcal{H}_\beta= - \frac{d^2}{dx^2}+x+ \frac{2}{\sqrt \beta} b'(x)
\end{equation}
where we take $b'$ to be a white noise.  A precise definition and many properties of this operator  can be found in \cite{RRV}.  We review the necessary ones below.

For our purposes it is sufficient to define an eigenfunction/eigenvalue pair in the following way:  Let 
\[L^*= \left\{f\in L^2[0,\infty)| f(0)=0 \text{ and } \int_0^\infty (f')^2+(1+x)f^2 dx<\infty\right\},\]
then $(\varphi ,\lambda)$ is an eigenvalue/eigenfunction pair for $\mathcal{H}_\beta$ if $\|\varphi\|_2=1, \varphi\in L^*$ and 
\begin{equation}
\varphi ''(x)= \frac{2}{\sqrt \beta} \varphi(x)b'(x)+(x-\lambda)\varphi(x)
\end{equation}
holds in the sense of distributions.  This may be written as 
\begin{equation}
\varphi'(x)- \varphi'(0)= \frac{2}{\sqrt \beta} \varphi(x)b(x)-\frac{2}{\sqrt \beta} \int_0^x \varphi'(t)b(t)dt+ \int_0^x(t-\lambda)\varphi (t)dt.
\end{equation}
In this sense, the set of eigenvalues is a deterministic function of the Brownian path $b$.  Moreover the eigenvalues are ``nice'' in the following sense: 

\begin{theorem}
\text{\cite{RRV}}
With probability one, the eigenvalues of $\mathcal{H}_\beta$ are distinct with no accumulation point, and for each $k\geq 0$ the set of eigenvalues of $\mathcal{H}_\beta$ has a well defined $(k+1)$st lowest element $\Lambda_k(\beta)$.
\end{theorem}

\noindent The Airy$_\beta$ point process is given by the eigenvalues of $\mathcal{H}_\beta$. Our first result shows that the spectrum of the Jacobi ensemble near to the soft edge converges to this point process after appropriate scaling and centering.

Let us introduce some notation:  Take 
\begin{align}
\label{cs} c^2& = \frac{n_1}{n_1+n_2}, & s^2&= \frac{n_2}{n_1+n_2}\\
\label{tildecs} \tilde c^2 & = \frac{n}{n_1+n_2}, & \tilde s^2 & = \frac{n_1+n_2-n}{n_1+n_2}
\end{align}
with $c,s, \tilde c$ and $\tilde s$ all nonnegative.  Under this notation we have that the expected edge of the spectrum is given by
\[ \Lambda_{\pm} = (c\tilde s\pm s \tilde c)^2.\]
We define our scaling factor to be
 \begin{equation}
 \label{alphan}
 \alpha_n=\frac{m^2_n}{ cs \tilde c \tilde s}, \hspace{1cm} \text{ where } \hspace{1cm} m_n = \left[ \frac{cs\tilde c \tilde s \sqrt{n_1+n_2}}{\tilde c \tilde s (c^2-s^2)+ c s (\tilde c^2-\tilde s^2)}\right]^{2/3}.
\end{equation} 

\noindent We now state the scaling limit near the soft edge.

\begin{theorem}\label{SEL}
Let $\lambda_1 \geq \lambda_2 \geq \cdots \geq \lambda_n$ denote the ordered eigenvalues of $J(\beta,n,n_1,n_2)$, and assume $\liminf n_2/n > 1$,  then
\[
\Big(\alpha_n(\Lambda_+ - \lambda_\ell)\Big)_{\ell=1,...,k} \Rightarrow \Big(\Lambda_0(\beta),...,\Lambda_{k-1}(\beta)\Big)
\]
jointly in law for any fixed $k<\infty$, as $n\to \infty$.
\end{theorem}

\begin{remark} \label{symmetry}
This theorem describes the limiting behavior of the $\beta$-Jacobi ensemble in the upper soft edge situation.  This is sufficient to determine the behavior in the lower soft edge because $J(n,n_1,n_2,\beta)$ is symmetric in $n_1$ and $n_2$.  That is, the reflection of the density of $J(n,n_1,n_2,\beta)$ with respect to $x=1/2$ is the corresponding density for $J(n,n_2,n_1,\beta)$.
\end{remark}


\begin{remark}
\label{mn bound}
In the situation where $n_1$ and $n_2$ are constant multiples on $n$, then $\alpha_n= c n^{2/3}$ for some constant $c$.  One can compare this with the scaling exponents in the Tracy-Widom result. 
In fact in the case where $\liminf n_2/n\geq 1+\epsilon$.  Then with a little bit of work one can show that there exist constants $c_1$ and $c_2$ depending only on $\epsilon$ so that $c_1 n^{1/3} \leq m_n \leq c_2n^{1/2}$.
\end{remark}

\begin{remark}
While we expect a soft edge type result for $\beta$-Jacobi spectrum in the case where $n_2=n+a_n$ and $a_n\to \infty$ we restrict  to the case where $\liminf n_2/n > 1$.  We make this restriction on $n_2$ because, as the previous remark suggests, the order of $m_n$ can change substantially which will render many computations invalid.  A similar gap exists in the $\beta$-Laguerre case for the lower soft edge.
\end{remark}




\subsection{Hard edge limit}\label{intro-hard}  The proof of the hard edge limit will follow largely from the proof of Ram\'irez and Rider in \cite{RR}.  We will formulate a more general theorem which gives conditions for convergence of the eigenvalues, and apply this in the specific setting of the $\beta$-Jacobi ensemble.  The main idea is again to embed the tridiagonal matrices as operators and show that operator convergence implies convergence of the eigenvalues.  However the operator convergence is not shown directly, instead we work with inverse operators to draw our conclusions.

For convenience and in analogy to the $\beta$-Laguerre ensemble, we will focus on the lower hard edge. We can see from Remark \ref{symmetry} that this is indeed sufficient to determine the upper edge as well.

We recall the definition of the stochastic Bessel operator, studied by Ram\'irez and Rider in \cite{RR} to describe the limit of the Laguerre ensemble in the hard-edge case.  The operator acts on functions on $\rr_+$ and is given by:
\[\mathfrak{G}_{\beta,a}=-\exp \left[ (a+1)x+ \frac{2}{\sqrt{\beta}} b(x) \right] \cdot \frac{d}{dx} \left( \exp\left[-ax- \frac{2}{\sqrt{\beta}}b(x)\right]\frac{d}{dx}\right),\]
where $b(x)$ is a Brownian motion, $a>-1$ and $\beta>0$. This can be rewritten as 
\[-\mathfrak{G}_{\beta,a}= e^x\left( \frac{d^2}{dx^2}-(a+\frac{2}{\sqrt{\beta}} b'(x))\frac{d}{dx}\right).\]

\noindent With Dirichlet boundary conditions at 0 and Neumann conditions at infinity the inverse operator as given by Ram\'irez and Rider \cite{RR} is:
\begin{equation}\label{inv-spectral-bessel}
( \mathfrak{G}_{\beta,a}^{-1} \psi)(x)\equiv \int_0^\infty\left(\int_0^{x \wedge y} e^{az+\frac{2}{\sqrt \beta}b(z)}dz\right)\psi(y) e^{-(a+1)y- \frac{2}{\sqrt \beta}b(y)}dy.
\end{equation}
The operator $\mathfrak{G}_{\beta,a}^{-1}$ is non-negative symmetric in $L^2[\rr_+,m]$ where
\[m(dx)= e^{-(a+1)x- \frac{2}{\sqrt \beta}b(x)}dx.\]
We may then define the eigenvalues and eigenfunctions of $\mathfrak{G}_{\beta,a}$ by taking the equation $\mathfrak{G}_{\beta,a} \psi= \lambda \psi$ to mean $ \psi = \lambda \mathfrak{G}_{\beta,a}^{-1} \psi$.  Moreover, it can be shown that the spectrum defines a simple point process as desired.

\begin{theorem}
\cite{RR}
With probability one, when restricted to the positive half-line with Dirichlet boundary condition (at the origin), $\mathfrak{G}_{\beta,a}$ has a discrete spectrum of simple eigenvalues $0<\Lambda_0(\beta,a)<\Lambda_1(\beta,a)<\cdots \uparrow \infty$.
\end{theorem}

We will show that the Bessel operator will describe the limiting spectrum of the $(\beta, n, n_1,n_2)$-Jacobi ensemble in the hard-edge case. 

\begin{theorem}\label{HEL}
Let $0<\lambda_0<\lambda_1<\cdots <\lambda_{n-1}$ be the ordered eigenvalues of  $J(\beta,n,n_1,n_2)$, $m_n=n n_2$ and $n_2>n$.  Assume that $(n_1-n)\to a \in (-1,\infty)$, then
\[\Big(m_n \lambda_0,m_n \lambda_1,...,m_n \lambda_{k}\Big) \Rightarrow \Big( \Lambda_0(\beta,a),\Lambda_1(\beta ,a), ..., \Lambda_{k}(\beta,a)\Big)\]
jointly in law, for any fixed $k<\infty$ as $n\to \infty$.
\end{theorem}

\begin{remark}
\label{coupling}
Previous work on both the hard and the soft edge of the $\beta$-Jacobi was done through a coupling with the $\beta$-Laguerre ensemble by Jiang \cite{Jiang}.  This work covers the case where $n / n_1 \to \gamma \in (0,1]$ and $n = o(\sqrt{n_2})$.  Our work extends the results to all cases for which $\liminf n_2/n >1$.
\end{remark}


\section{Convergence of the spectrum at the soft edge}\label{se}

We begin by stating a general result of Ram\'irez, Rider, and Virag \cite{RRV} giving conditions for a sequence of operators to converge to the stochastic Airy operator in an appropriate sense. Our work will then consist of verifying the hypothesis of this theorem in the $\beta$-Jacobi case.

Let $H_n:\rr^n \to \rr^n$ be the linear operator whose associated matrix with respect to the standard bases is symmetric, tridiagonal with diagonal entries $(2m_n^2+m_n(y_{n,1,k}-y_{n,1,k-1}), k\geq 1)$ and off-diagonal entries $(-m_n^2+m_n(y_{n,2,k}-y_{n,2,k-1})/2,k\geq 1)$.

We define define the step functions $y_{n,i}(x)=y_{n,i,\lfloor xm_n \rfloor} \ind_{x m_n \in [0,n]}$ and make the following assumptions:

\noindent {\it Assumption 1 (Tightness/Convergence)} There exists a continuous process $x\mapsto y(x)$ such that 
\begin{eqnarray*}
\big(y_{n,i}(x); x\geq 0\big) && i=1,2 \ \ \text{ are tight in law}\notag \\
\big(y_{n,1}(x)+y_{n,2}(x);x \geq 0\big) & \Rightarrow &\big(y(x); x\geq 0\big) \ \ \text{ in law,}\notag
\end{eqnarray*}
with respect to the Skorokhod topology.

\noindent{\it Assumption 2 (Growth/Oscillation bound)} There is a decomposition
\begin{equation}
y_{n,i,k}= m_n^{-1} \sum_{\ell=1}^k \eta_{n,i,\ell}+ \omega_{n,i,k}
\end{equation}
with $\eta_{n,i,k}\geq 0$, such that there are deterministic unbounded nondecreasing continuous functions $\overline{\eta}(x)>0, \zeta(x) \geq 1,$ and random constants $\mu_n(\omega)\geq 1$ defined on the same probability space which satisfy the following:  The $\mu_n$ are tight in distribution, and, almost surely
\begin{eqnarray}
\overline{\eta}(x)/\mu_n - \mu_n \leq \eta_{n,1}(x)+ \eta_{n,2}(x) &\leq& \mu_n(1+ \overline{\eta}(x)) \label{tightness1}\\
\eta_{n,2}(x) & \leq & 2m_n^2\\
\label{tightness2}
|\omega_{n,1}(\xi)-\omega_{n,1}(x)|^2+ |\omega_{n,2}(\xi)-\omega_{n,2}(x)|^2& \leq & \mu_n(1+ \overline{\eta}(x)/\zeta (x)) \label{oscillation}
\end{eqnarray}
for all $n$ and $x,\xi \in [0,m_n]$ with $|x-\xi|\leq 1$.

\begin{theorem} \cite{RRV}
\label{general soft}
Given Assumptions 1 and 2 above and any fixed $k$, the bottom $k$ eigenvalues of the matrix $H_n$ converge in law to the bottom $k$ eigenvalues of the operator $H$, where
\[H=- \frac{d^2}{dx^2}+ y'(x).\]
\end{theorem}

Here the eigenfunction/eigenvalue pairs of $H$ should be understood in the same way as those of $\mathcal{H}_\beta$ as shown in Section \ref{intro-soft}.  We will show that a tridiagonal matrix with eigenvalues corresponding to the  $\beta$-Jacobi ensemble satisfies the requirements of the theorem with limiting operator $\mathcal{H}_\beta$.


\subsection{The $\beta$-Jacobi model}

We use the following tridiagonal model for the $\beta$-Jacobi ensemble.  Let
\[
Z_{n,\beta}=\left[
\begin{array}{ccccc}
C_n &  S_{n} \tilde C_{n-1}&&&\\
&C_{n-1} \tilde S_{n-1}&S_{n-1} \tilde C_{n-2}&&\\
\\
&&\ddots&\ddots&\\
\\
&&&C_2 \tilde S_2& S_1 \tilde C_2\\
&&&&C_1 \tilde S_1
\end{array}
\right]
\]
with $C_k,S_k, \tilde C_k,$ and $\tilde S_k$ defines as in Theorem \ref{tridiagonal}.  Clearly $\sigma(Z_{n,\beta}Z_{n,\beta}^T)=\sigma (M_{n,\beta}M_{n,\beta}^T)$ in distribution. 

 Recall the definitions of $m_n$ and $\alpha_n$ from section \ref{intro-soft} given in (\ref{alphan}).  We consider the matrix 
\[H_n= \alpha_n \Big( (c\tilde s+s\tilde c)^2I_n- Z_{n,\beta}Z_{n,\beta}^T\Big).\]
This matrix has diagonal and off-diagonal entries given respectively by 
\[2m_n^2+\frac{m_n^2}{cs\tilde c\tilde s} (c^2\tilde s^2+s^2\tilde c^2- S_{n-k+1}^2\tilde C_{n-k}^2-C_{n-k}^2\tilde S_{n-k}^2)\]
and
\[-m_n^2+\frac{m_n^2}{cs\tilde c\tilde s} (cs\tilde c\tilde s- C_{n-k}S_{n-k} \tilde C_{n-k}\tilde S_{n-k-1}).\]
Note that the term $(c\tilde s+s\tilde c)^2$ is the $\Lambda_+$ defined in the introduction.  We define $\Delta y_{n,i,k}= y_{n,i,k}-y_{n,i,k-1}$ with

\begin{eqnarray*}
\Delta y_{n,1,k}&=& \frac{m_n}{cs\tilde c\tilde s} (c^2\tilde s^2+s^2\tilde c^2- S_{n-k+1}^2\tilde C_{n-k}^2-C_{n-k}^2\tilde S_{n-k}^2)\\
\Delta y_{n,2,k}&=&\frac{2m_n}{cs\tilde c\tilde s} (cs\tilde c\tilde s- C_{n-k}S_{n-k} \tilde C_{n-k}\tilde S_{n-k-1}).
\end{eqnarray*}

\noindent The proof now consists of verifying the hypothesis of Theorem \ref{general soft}.


\subsection{Checking Assumption 1}

To show that $H_n$ satisfies Assumption 1 of Theorem \ref{general soft} we use the following proposition which is a simple modification of Theorem 7.4.1 and Corollary 7.4.2 in \cite{EK}.

\begin{proposition}[\cite{EK}]
\label{kurtz}
Let $f\in C^1(\rr^+)$ and $g\in C^1(\rr^+)$, and let $y_n$ be a sequence of processes with $y_{n,0}=0$ and independent increments.  Assume that 
\[\frac{1}{\epsilon_n}E(\Delta y_{n,k})= f'(k \epsilon_n)+o(1), \hspace{.8cm} \frac{1}{\epsilon_n}\var(\Delta y_{n,k})= g^2(k \epsilon_n)+o(1), \hspace{.8cm} \frac{1}{\epsilon_n}E(\Delta y_{n,k})^4=o(1)\]
uniformly for $k \epsilon_n$ on compact sets as $n\to \infty$.  Then $y_n(t)=y_{n,\lfloor t/\epsilon_n \rfloor}$ converges in law, with respect to the Skorokhod topology, to the process $f(t)+\int_0^t g(s)db_s$, where $b$ is a standard Brownian motion.

\end{proposition}

\noindent We take $\epsilon_n=1/m_n$ and apply this to a slightly altered version of $y_{n,1}$ and $y_{n,2}$. Take

\begin{eqnarray*}
\Delta \tilde  y_{n,1,k}&=& \frac{m_n}{cs\tilde c\tilde s} (c^2\tilde s^2+s^2\tilde c^2- S_{n-k}^2\tilde C_{n-k}^2-C_{n-k}^2\tilde S_{n-k}^2)\\
\Delta  \tilde y_{n,2,k}&=&\frac{2m_n}{cs\tilde c\tilde s} (cs\tilde c\tilde s- C_{n-k}S_{n-k} \tilde C_{n-k}\tilde S_{n-k}).
\end{eqnarray*}

\noindent Computation gives
\[ ES_{n-k}^2= \frac{ n_2-k}{n_1+n_2-2k}=s^2+ \frac{s^2-c^2}{n_1+n_2}+\frac{2(n_2-n_1)}{(n_1+n_2-2\ell)^3}k^2 \]
for some $\ell \in [0,k]$.  Similar expansions can be written for the other terms involved.  We now note that for convergence on compact subsets it is sufficient to consider $k\leq c m_n \leq c_2 n^{1/3}$ by remark \ref{mn bound}.  Therefore, collecting terms we find that

\begin{equation}
\label{first-1}
m_n E \Delta \tilde  y_{n,1,k} = \frac{m_n^2}{cs\tilde c\tilde s} \cdot \frac{2(c^2-s^2)(\tilde c^2-\tilde s^2)}{n_1+n_2} \ k + o(1).
\end{equation}

\begin{lemma}
Let $X= C_{n-k}^i S_{n-k}^j$ with $i$ and $j$ positive integers, then
\[E X= \sqrt{EX^2}- \frac{1}{8 (EX^2)^{3/2}} \var (X^2)+ O\left( \frac{1}{(n_1+n_2)^2}\right).\]
Similar statements hold for $\tilde C_{n-k}$ and $\tilde S_{n-k}$.
\end{lemma}

\noindent Here we take $O(1/(n_1+n_2)^2)$ to mean that there exists a constant $C$ depending only on $i$ and $j$ so that the magnitude of the error is bounded about by $C/(n_1+n_2)^2$.  This lemma will always be used in the event that $C_{n-k}$ or $S_{n-k}$ is raised to an odd power.

\begin{proof}
For $a<x\in [0,1]$
\[ \sqrt{a} + \frac{1}{2\sqrt{a}}(x-a) - \frac{1}{8 a^{3/2}}(x-a)^2 \leq \sqrt{x} \leq \sqrt{a} + \frac{1}{2\sqrt{a}}(x-a) - \frac{1}{8 a^{3/2}}(x-a)^2
 + \frac{ (x-a)^3}{16a^{5/2}},\]
 and for $x\leq a\in [0,1]$
 \[\sqrt{a} + \frac{1}{2\sqrt{a}}(x-a) - \frac{1}{8 a^{3/2}}(x-a)^2
 + \frac{ (x-a)^3}{16a^{5/2}} \leq \sqrt x \leq \sqrt{a} + \frac{1}{2\sqrt{a}}(x-a) - \frac{1}{8 a^{3/2}}(x-a)^2.\]
   To complete the proof what remains to be shown is that $E(X^2-EX^2)^3=O(1/(n_1+n_2)^2)$.  
To accomplish this we make the following observation.  If $Y\sim \text{Beta}(a,b)$, then for $i,j$ positive integers we have that

\begin{equation}
\label{beta ex}
E(Y^i(1-Y)^j)= \frac{(a+i-1)\cdots (a+1)a \, (b+j-1) \cdots (b+1)b}{(a+b+i+j-1)\cdots (a+b+1)(a+b)}
\end{equation}
direct computation then finishes the proof.
\end{proof}

\noindent An application of this lemma shows that

\begin{equation}
m_n E\Delta \tilde y_{n,2,k}= \frac{m_n^2}{cs\tilde c\tilde s} \cdot \frac{\tilde c^2 \tilde s^2(c^2-s^2)^2+c^2s^2 (\tilde c^2-\tilde s^2)^2}{ n_1+n_2}k +O\left(\frac{m_n^2}{n_1+n_2}\right)
\end{equation}
Here we mean that for $1\leq k \leq  a m_n$ there exists some constant $C$ independent of $k, n, n_1,n_2$ so that the error is bounded by $Cm_n^2/(n_1+n_2)$. 
\noindent Together with (\ref{first-1}) this gives us that

\begin{equation}
m_nE(\Delta \tilde y_{n,1,k} + \Delta \tilde y_{n,2,k})= \frac{k}{m_n}+o(1)
\end{equation}

\noindent which verifies the first condition in Proposition \ref{kurtz} with $f(x)=x^2/2$.

\noindent The second and fourth moment computations can be done similarly.  They give us that 

\begin{equation}
m_nE(\Delta \tilde y_{n,1,k} + \Delta \tilde y_{n,2,k})^2= \frac{4}{\beta}+o(1),  \ \ \ \text{and } \ \ \ 
m_nE(\Delta \tilde y_{n,1,k} + \Delta \tilde y_{n,2,k})^4= o(1).
\end{equation}

\noindent Note that, for the fourth moment, it is sufficient to bound $m_nE(\Delta \tilde y_{n,2,k})^4$ and (breaking $\tilde y_{n,1,k}$ into pieces) to bound the corresponding moments of $s^2\tilde c^2- S_{n-k}^2\tilde C_{n-k}^2$ and $c^2\tilde s^2- C_{n-k}^2 \tilde S_{n-k}^2$.  Therefore, Proposition \ref{kurtz} can be applied and yields
\[ \tilde y_{n,1,k}+\tilde y_{n,2,k} \Rightarrow \frac{x^2}{2}+ \frac{2}{\sqrt{\beta}} b(x),\]
in law in the Skorohod topology.

\noindent Recall that in all the above computations, we considered $\tilde y$ instead of $y$. Consequently, all that remains to check the second part of Assumption 1 is to show that $y_{n,1,k}+y_{n,2,k}- \tilde y_{n,1,k}- \tilde y_{n,2,k}$ converges to the $0$ process in law in the Skorohod topology.  One can check that the expectation and variance of the increments are of order $1/n$ and so the expectation and variance of the process go to $0$ on compact subsets.  This together with a fourth moment bound gives us convergence to the $0$ process in law.

Finally, the individual tightness of each $(y_{n,i}(x);x\geq 0)$, $i=1,\, 2$ can be obtained along the same lines.


\subsection{Checking Assumption 2}

To check this assumption we again work with the shifted processes $\tilde y_{n,i,k}$ and compare this with the original process.  To this end we take $\eta_{n,i,k}= m_n E\Delta y_{n,i,k}$, and similarly $\tilde \eta_{n,i,k}= m_n E\Delta \tilde y_{n,i,k}$.  Further, we will neglect the extra $-1$ found in the $\tilde C_{n-k}$ and $\tilde S_{n-k}$ terms and show the irrelevance later.  We will show the inequality in (\ref{tightness1}) for $\bar \eta (x)=x$.  Under these definitions we have

\begin{align*}
\tilde \eta_{n,1,k} &=  \frac{m_n^2}{cs \tilde c \tilde s} \left( c^2\tilde s^2 + s^2 \tilde c^2 - \frac{(n_1-k)(n_1+n_2-n-k)}{(n_1+n_2-2k)^2} - \frac{(n_2-k)(n-k)}{(n_1+n_2-2k)^2} \right)\\
\tilde \eta_{n,2,k} &=\frac{2m_n^2}{cs \tilde c \tilde s} \left( cs\tilde c\tilde s - \frac{f (\frac{\beta}{2}(n_1-k))f(\frac{\beta}{2}(n_2-k))f(\frac{\beta}{2}(n-k))f( \frac{\beta}{2} (n_1+n_2-n-k))}{\frac{\beta^2}{4}(n_1+n_2-2k)^2 }\right)
\end{align*}
where $f(x)=\Gamma(x+1/2)/\Gamma(x)$.

We begin with the upper bound on $\tilde \eta_{n,1,k}+\tilde \eta_{n,2,k}$.  The general idea will be to treat $\tilde \eta_{n,1,k}$ via a Taylor expansion with a uniform bound on the error term, and work with  $\tilde \eta_{n,2,k}$ in two regions.  The first region will be $1\leq k \leq \alpha n$ for some $\alpha<1$ and the second will be $\alpha n \leq k \leq n-1$.  On the first region we will again work with primarily with Taylor expansion and find a uniform bound on the error in terms of $\alpha$, for the second section we return to working with the original $\tilde \eta_{n,2,k}$ and show that for $\alpha n \leq k \leq n-1$ this can be bounded by $C_\alpha k/m_n$.

Using the inequality
\[  \sqrt{x}\left( 1- \frac{2}{x}\right) \leq \frac{ \Gamma(x+1/2)}{\Gamma(x)} \leq \sqrt{x} \] 
we get that for $n_1, n_2 >n>k$
\begin{align*}
\tilde \eta_{n,2,k} \leq  \frac{2m_n^2}{cs \tilde c \tilde s}& \left( cs\tilde c\tilde s - \frac{ \sqrt{(n_1-k)(n_2-k)(n-k)(n_1+n_2-n-k)}}{(n_1+n_2-2k)^2}\right)\\
&\ \ \ + 15 \frac{4m_n^2}{cs \tilde c \tilde s} \left( \frac{ \sqrt{(n_1-k)(n_2-k)(n-k)(n_1+n_2-n-k)}}{\frac{\beta}{2}(n-k)(n_1+n_2-2k)^2}\right).
\end{align*}
For convenience we will label the first line of the right hand side by $A_k$ and the second line by $B_k$.  We will treat the second term $B_k$ first.
\[ \frac{m_n B_k}{ k} \leq \frac{120}{\beta}\frac{n_1 n_2 n(n_1+n_2-n)(n_1+n_2)^3}{n(n_1+n_2-2n)^2 (\sqrt{n (n_1+n_2-n)}(n_1-n_2)+\sqrt{n_1n_2}(2n-n_1-n_2))^2}\]
This upper bound has a finite limsup.  We now turn to $\tilde \eta_{n,1,k}+A_k$ for $1 \leq k \leq \alpha n$ for some $0<\alpha<1$.
\begin{align*}
\tilde \eta_{n,1,k}+A_k& = \frac{k}{m_n} +\frac{ f(\ell)}{2}k^2
\end{align*}
for some $\ell \in [1,\alpha n]$.  Here $f(\ell)$ is the second derivative of $\tilde \eta_{n,1,k} +A_k$ with respect to $k$.
On this range of $k$ we can find explicit upper bounds for 
\[m_n k  \frac{f(\ell)}{2}\]
 in terms of $n, n_1, n_2,$ and $\alpha$ which have finite limsup as $n$ goes to $\infty$.  This together with our bound on $B_k m_n/k$ is enough to give us a sequence of constants $c_n$ so that 
\[ \tilde \eta_{n,1,k} + \tilde \eta_{n,2,k} \leq c_n \frac{k}{m_n}\]
for $1\leq k \leq \alpha n$.  For the remaining piece we work with $\tilde \eta_{n,1,k}$ and  $\tilde \eta_{n,2,k}$ separately.  
\begin{align*}
\frac{m_n \tilde \eta_{n,1,k}}{k} &= \frac{m_n^3}{cs \tilde c \tilde s} \left( \frac{2(n_1-n_2)(2n-n_1-n_2)}{(n_1+n_2)^3}+\frac{6(n_1-n_2)(2n-n_1-n_2)k}{(n_1+n_2-2\ell)^4}\right)
\end{align*}
for some $1\leq \ell \leq k \leq n$.  Taking the obvious upper bound with $\ell = k=n$ we have an upper bound on $\tilde \eta_{n,1,k}$ for $1\leq k \leq n$ which has a finite limsup.  Returning to $A_k$ we note that for $\alpha n \leq k \leq n$ we have that
\[ \frac{m_n \tilde \eta_{n,2,k}}{k} \leq  \frac{m_n^3}{\alpha n}\]
From Remark \ref{mn bound} we have that this is an appropriate upper bound.  Choosing the larger upper bound on the two sections gives us a sequence $\mu_n$ such that
\[ \tilde \eta_{n,1,k} + \tilde \eta_{n,2,k} \leq \mu_n \frac{k}{m_n}\]
for $1 \leq k \leq n-1$.

Turning to the lower bound we have the inequality
\begin{align} 
\tilde \eta_{n,1,k}+ \tilde \eta_{n,2,k} \geq & \frac{m_n^2}{cs\tilde c\tilde s}\left(c^2\tilde s^2+s^2\tilde c^2-\frac{(n_1-k)(n_1+n_2-n-k)}{(n_1+n_2-2k)^2}-\frac{(n_2-k)(n-k)}{(n_1+n_2-2k)^2}\right)\notag \\
 & \ \ \ +\frac{2m_n^2}{cs\tilde c\tilde s}\left(cs\tilde c\tilde s- \frac{ \sqrt{(n_1-k)(n_2-k)(n-k)(n_1+n_2-n-k)}}{(n_1+n_2-2k)^2}\right).
 \label{lower}
 \end{align}
One can then make arguments for $1\leq k \leq \alpha n$ similar to those employed in the upper bound, simply choose $\alpha$ small enough so that $m_n k  f(\ell)/2$ is bounded below by $-1$.  For  the remaining $k$ we note that the derivative of the right hand side with respect to $k$ is
\[\frac{m_n^2}{cs\tilde c\tilde s} \left[ \frac{\sqrt{(n-k)(n_1+n_2-n-k)}(n_1-n_2)+\sqrt{(n_1-k)(n_2-k)}(2n-n_1-n_2)}{\sqrt{(n_1-k)(n_2-k)(n-k)(n_1+n_2-n-k) (n_1+n_2)}}\right]^2\]
which is strictly greater then $0$ for $1\leq k \leq n-1$.  Therefore 
\begin{align*} 
\tilde \eta_{n,1,k}+ &\tilde \eta_{n,2,k} \geq \\
& \frac{m_n^2}{cs\tilde c\tilde s}\left(c^2\tilde s^2+s^2\tilde c^2-\frac{(n_1-\alpha n)(n_1+n_2-n-\alpha n)}{(n_1+n_2-2\alpha n)^2}-\frac{(n_2-\alpha n)(n-\alpha n)}{(n_1+n_2-2\alpha n)^2}\right) \\
 & \ \ \ \ \  +\frac{2m_n^2}{cs\tilde c\tilde s}\left(cs\tilde c\tilde s- \frac{ \sqrt{(n_1-\alpha n)(n_2-\alpha n)(n-\alpha n)(n_1+n_2-n-\alpha n)}}{(n_1+n_2-2\alpha n)^2}\right).
 \end{align*}
This lower bound can be used to get the desired constants for $\alpha n \leq k \leq n-1$ finishing the lower bound.  To finish we make the following observation:  By direct computation we can find an upper bound $\delta_n$ of order $(n_1+n_2)^{-1/3}$ such that
\[|\tilde \eta_{n,1,k}+\tilde \eta_{n,2,k} - \eta_{n,1,k}- \eta_{n,2,k}|\leq \delta_n.\]
Therefore any error that comes from neglecting the $-1$ in the $\tilde C_{n-k}$ and $\tilde S_{n-k}$ terms, or working with the shifted process may be absorbed into the constant terms.  This finishes the proof of the bound in (\ref{tightness1}).

 \noindent We now verify (\ref{tightness2}):
 \[ \Delta y_{n,2,k} = \frac{2m_n}{cs\tilde c\tilde s}(cs\tilde c\tilde s- C_{n-k}S_{n-k}\tilde C_{n-k}\tilde S_{n-k-1})= 2m_n - \frac{2m_n C_{n-k}S_{n-k}\tilde C_{n-k}\tilde S_{n-k-1}}{cs\tilde c\tilde s}. \]
Therefore since all of the relevant pieces are positive we have that 
\[ m_n E\Delta y_{n,2,k} = 2m_n^2- E\frac{2m_n^2 C_{n-k}S_{n-k}\tilde C_{n-k}\tilde S_{n-k-1}}{cs\tilde c\tilde s}\leq 2m_n^2.\]

\noindent The proof of the oscillation bound (\ref{oscillation}) is identical to the corresponding proof in \cite{RRV} and follows from general martingale arguments. This ends the proof of the convergence at the soft edge.



\section{Convergence of the spectrum at the hard edge}\label{he}

To show the convergence of the spectrum at the hard edge we will first state a more general theorem on convergence of eigenvalues.  We will then show that this gives us Theorem \ref{HEL}.  Though this is stated more generally then the result given by Ram\'irez and Rider in \cite{RR} the majority of the proof follows directly from their work.

\subsection{A more general setting}

Let $X_n$ be a lower bidiagonal matrix 
\[ X_n = \left[ \begin{array}{ccccc} a_1 &&&& \\
-b_1 & a_2 &&& \\
& -b_2 & a_3 & & \\
&& \ddots & \ddots & \\
&&& -b_{n-1} & a_n \end{array} \right]\]
We make the following assumptions on the entries:

\noindent{\it Assumption 1}
There is a brownian motion $B(\cdot)$, and functions $r(x)$, $s(x)$ and $\varphi(x)$ continuous on $(0,1)$ such that for $y<x$ in $(0,1]$
\begin{align}
\frac{n}{a_{\lfloor nx \rfloor} }& \Rightarrow r(x) \\
\sum_{k=\lfloor ny \rfloor}^{\lfloor nx \rfloor}  \log \left( \frac{b_k}{a_k}\right) & \Rightarrow s(x)-s(y) + \int_y^x \varphi (t) dB_t .
\end{align}

.

\smallskip

\noindent{\it Assumption 2}  There exist tight random constants $\kappa_n$ and $\kappa_n'$ such that 
\begin{align}
\label{A2mainterm}
\sup_{1\leq k \leq n} \frac{E a_k}{a_k} & \leq \kappa_n\\
\label{A2process}
\sum_{k=j}^{i-1} \log  \left( \frac{b_k}{a_k}\right) - s(i/n)+s( j/n) & \leq \kappa_n' ( 1+ T^{3/4}(i/n)+T^{3/4}(j/n))
\end{align}
Where $T(x) =  \int_{1/2}^x \varphi^2(t) dt$.

\smallskip

\noindent{\it Assumption 3}  Let $K_C$ to be the integral operator with kernel
\[ k_C(x,y) = C r(x) e^{s(x)-s(y)} \exp \left[ C T^{3/4}(x)+ CT^{3/4}(y)\right] \ind(y<x),\]
then $K_C$ is Hilbert-Schmidt.

\smallskip

\noindent{\it Assumption 4} Let $K$ be the integral operator  with kernel given by 
\[k(x,y) = r(x) e^{s(x)-s(y)} \exp \left[ \int_y^x \varphi(t)dB_t\right] \ind (y<x).\]
The operator $(KK^T)^{-1}$ has discrete spectrum with simple eigenvalues $0<\Lambda_0< \Lambda_1< ... \uparrow \infty$.

\medskip

\begin{theorem}
\label{inverselimit}
Let $X_n$ be a lower bidiagonal matrix with entries that satisfy assumptions 1, 2, 3 and 4.
Denote the ordered eigenvalues of $X_nX_n^T$ by $\lambda_0< \lambda<1< ...$, then 
\[ \left( \lambda_0, \lambda_1,..., \lambda_{k-1}\right) \Rightarrow \left( \Lambda_0, \Lambda_1,...,\Lambda_{k-1}\right)\]
jointly in law in the Skorokhod topology for any fixed $k$ as $n\to \infty$.

\end{theorem}

\begin{remark}
Assumption 4 can be replaced by further conditions on $r(x),s(x)$ and $\varphi(x)$ in assumption 1.  We omit this here because in our case assumption 4 can be checked directly.
\end{remark}

\begin{proof}
We begin by computing $X_n^{-1}$ and embedding the resulting matrix as an operator on $L^2(0,1]$.  By Lemma 4 in \cite{RR} we can compute 
\[[X_n^{-1}]_{i,j} = \frac{1}{a_i} \prod_{k=j}^{i-1} \frac{b_k}{a_k} \hspace{.5cm} \text{ for } \  j\leq i.\]
The action of our operator on $L^2$ will then read
\[\left( X_n^{-1} f \right) (x) = \frac{n}{a_{\lfloor nx \rfloor}} \sum_{j=1}^{\lfloor nx\rfloor}  \prod_{k=j}^{\lfloor nx \rfloor -1} \frac{ b_k}{a_k} \int_{x_{j-1}}^{x_j} f(x) dx\]
where $x_k=k/n$.   We denote this operator by $K^n$, and note that this is a discrete integral operator with kernel 
\[k^n(x,y)= \frac{1}{a_i} \exp \left[ \sum_{k=1}^{i-1} \log \left( \frac{b_k}{a_k}\right) \right] \ind_L(x,y),\]
where $\ind_L(x,y)= \ind(x\in [x_{i-1},x_i)) \ind (y\in [x_{j-1},x_j))$ and $i>j$.  

To complete the proof we need the following lemmas:

\begin{lemma}
\label{operator convergence}
There exists a probability space on which all $K^n$ and $K$ are defined, and such that any sequence of the operators $K^n$ contains a subsequence which converges to $K$ in Hilbert-Schmidt norm with probability one.  In particular for any $n_k \uparrow \infty$ we can find a subsequence $n_{k'}  \uparrow \infty$ along which
$$\lim_{n_{k'}\to \infty} \int_0^1 \int_0^1 |k^{n_{k'}}(x,y)(\omega)- k(x,y)(\omega)|^2 dx dy=0$$
almost surely.
\end{lemma}

The remainder of the proof of Theorem \ref{inverselimit} now follows from the proof of Theorem 1 in \cite{RR}.

\end{proof}

\begin{proof}[Proof of Lemma \ref{operator convergence}]
We make the following observation:  The noise term in the limiting process may be rewritten as 
\begin{align*}
\int_x^y \varphi(z) db_z = \tilde b \left(T(x)  \right) -\tilde b \left( T(y) \right).
\end{align*}
Here the equality is in distribution with a different Brownian motion $\tilde b$ living on the same probability space. This can be used to show that limiting operator $K$ is almost surely Hilbert-Schmidt.  The remainder of the proof then follows from the proof of Lemma 6 in \cite{RR}.

\end{proof}


\subsection{Proof of theorem \ref{HEL}}
\label{HEproof}

We now apply Theorem \ref{inverselimit} to our setting.  Recall the tridiagonal ensemble of random matrices which spectrum corresponds to the $\beta$-Jacobi ensemble. We look at a similar matrix with unchanged spectrum, which can be decomposed into a product of two bidiagonal matrices. Specifically we take

\[W_{n,\beta}= \left[ \begin{array}{ccccc} C_1 \tilde S_1 & &&& \\
-S_2\tilde C_1& C_2 \tilde S_2 &&&\\
& -S_3 \tilde C_2 & \ddots && \\
&& \ddots && \\
&&& -S_n \tilde C_{n-1} & C_n \end{array}\right],\]
where the random variables $C_k$ and $\tilde C_k$ are again distributed as in Section \ref{intro-tri}.  Clearly $\sigma(W_{n,\beta} W_{n,\beta}^T)= \sigma(M_{n,\beta} M_{n,\beta}^T)$.  Recall that we are looking for the hard edge limit at the lower edge, therefore we will take $n_1=n+a_n, a_n \to a\in(-1,\infty)$ with no restriction on $n_2$ beyond $n_2\geq n$. 

We will apply Theorem \ref{inverselimit} to the bidiagonal matrix $\sqrt{m_n} W_{n,\beta}$ where $m_n =n n_2$,  and then show that the eigenvalue equation $\varphi = \lambda (KK^T)^{-1} \varphi$ is equivalent to $\psi = \lambda \mathfrak{G}_{\beta,a}^{-1} \psi$ completing the proof of Theorem \ref{HEL}.

In order to apply Theorem \ref{inverselimit} we need to show that $\sqrt{m_n} W_{n,\beta}$ satisfies the assumptions.  All these assumptions will need to be verified in two cases.  The first case will be when $n_2/n \to \gamma \in [1,\infty)$, and the second will be $n_2 \gg n$.  Note that these two cases are sufficient because in the general case for any subsequence we can find a further subsequence along which either $n_2/n\to \gamma$ or $n_2\gg n$ and so convergence of to the eigenvalues of $\mathfrak{G}_{\beta,a}$ in those situations is sufficient. For clarity we note that the discrete kernel of $K^n$ is 
\[k^n_{\beta}(x,y)= \frac{n}{\sqrt{m_n} C_i \tilde S_i} \exp \left[ \sum_{k=j}^{i-1} \log \left(\frac{S_{k+1}\tilde C_k}{ C_{k} \tilde S_k}\right)\right]\ind_L(x,y),\]
with $\ind_L(x,y)= \ind_{x\in [x_{i-1},x_i)} \ind_{y\in [x_{j-1},x_j)}$.

\subsubsection{Verifying Assumption 1}

\begin{lemma}
\label{mainterm}
Assume that $n_2/n \to \gamma \in [1,\infty)$, then for $x\in (0,1]$
\[ \frac{n}{\sqrt{m_n} C_{\lfloor nx\rfloor} \tilde S_{\lfloor nx \rfloor}} \Rightarrow  \frac{2x+\gamma-1}{\sqrt \gamma \sqrt{x(x+\gamma-1)}}.\]
Assume that $n_2\gg n$ then for $x\in (0,1]$
\[ \frac{n}{\sqrt{m_n} C_{\lfloor nx\rfloor} \tilde S_{\lfloor nx \rfloor}} \Rightarrow  \frac{1}{\sqrt x}.\]
Both convergence statements hold in the Skorokhod topology.
\end{lemma}

\begin{proof}
We begin by observing that 
$$\var C^2_{nx} =O(1/n) \to 0, \ \ \ \text{ as } \ \ \ n\to \infty.$$
And similarly $\var \tilde S^2_{nx} =O(1/n)$, therefore the random variables converge to their expected values in $L^2$.  Now to ensure convergence on the process level fix $\delta>0$, then for $k>nx$, $x\geq \delta$ we can find a constant $c$ such that  
$$E\left( \frac{1}{C_{k+1}}\cdot \frac{\sqrt{a+k+1}}{\sqrt{a+n_2-n+2k+2}}-\frac{1}{C_k}\cdot \frac{\sqrt{a+k}}{\sqrt{a+n_2-n+2k}}\right)^2\leq \frac{c}{n}$$
and
$$E\left( \frac{1}{\tilde S_{k+1}}\cdot \frac{\sqrt{a+n_2-n+k+1}}{\sqrt{a+n_2-n+2k+2}}-\frac{1}{\tilde S_k}\cdot \frac{\sqrt{a+n_2-n+k}}{\sqrt{a+n_2-n+2k}}\right)^2\leq \frac{c}{n}.$$
Kolmogorov's tightness criterion ensures that we have process convergence on the set $[\delta, 1]$.  Let $\delta \to 0$ to finish the proof.

\end{proof}

\begin{lemma}
\label{process}
Assume that $n_2/n\to \gamma\in [1,\infty)$.  There is a Brownian motion $b(\cdot)$ such that for every $y,x\in (0,1]$ with $y<x$ we have 
\begin{align}
\label{process1}
\sum_{k= \lfloor ny \rfloor}^{\lfloor nx\rfloor -1}  \log  \left(\frac{S_{k+1}\tilde C_k}{ C_{k} \tilde S_k}\right)\Rightarrow \log& \left( \frac{\sqrt{x(\gamma-1+2y)}}{\sqrt{y(\gamma-1+2x)}}\right) + \frac{1+a}{2} \log \left( \frac{y}{x} \right)\\
& \ \ \  \notag + \frac{a}{2} \log \left( \frac{y+\gamma-1}{x+\gamma-1}\right)+ \int_y^x \frac{\sqrt{2s+\gamma-1}}{\sqrt{\beta(s^2+\gamma s-s)}}db_s
\label{process1}
\end{align}
Assume that $n_2 \gg n$.  There  is a Brownian motion $b(\cdot)$ such that for every $y,x\in (0,1]$ with $y<x$ we have 
\begin{equation}
\label{process2}
\sum_{k=\lfloor ny \rfloor }^{\lfloor nx \rfloor-1} \log \left(\frac{S_{k+1}\tilde C_k}{ C_{k} \tilde S_k}\right) \Rightarrow + \frac{a}{2}\log\left(\frac{y}{x}\right)+\frac{1}{ \sqrt{\beta}} \int_y^x \frac{db_s}{\sqrt s}
\end{equation}
again with both convergence statements holding in the Skorohod topology.
\end{lemma}

Before beginning the proof we note the following:

\begin{proposition}
Let $X\sim \text{Beta}(p,q)$, then 
\begin{align}
\notag
 \textup{E}\left(\log \sqrt{\frac{X}{1-X}}\right) = \frac12(\Psi_0(p)-\Psi_0(q)), \hspace{1cm}   \textup{Var}\left(\log \sqrt{\frac{X}{1-X}}\right) = \frac14(\Psi_1(p)+\Psi_1(q)),
\end{align}
where $\Psi_0$ and $\Psi_1$ are respectively the digamma and trigamma function. Moreover as $x\to \infty$
\begin{eqnarray}\nonumber
 \Psi_0(x) &=& \frac{\Gamma'(x)}{\Gamma(x)} = \log(x) - \frac{1}{2x} - \frac{1}{12x^2} + O(x^{-4}),\\ \label{gamma-exp}
\Psi'(x)&=& \Psi_0'(x)=\frac{1}{x}+\frac{1}{2 x^2}+O(x^{-3}).
\end{eqnarray}

\end{proposition}

\begin{proof}[Proof of Lemma \ref{process}]
We rearrange the process to have independent increments and then apply Proposition \ref{kurtz} with $\epsilon_n=1/n$.
\begin{align*}
\sum_{k=\lfloor ny \rfloor }^{\lfloor nx \rfloor-1} \log \left(\frac{S_{k+1}\tilde C_k}{ C_{k} \tilde S_k}\right)  &=\log C_{\lfloor nx \rfloor} - \log C_{\lfloor ny \rfloor}+ \sum_{k=\lfloor ny \rfloor }^{\lfloor nx \rfloor-1} \log \left(\frac{S_{k+1}\tilde C_k}{ C_{k+1} \tilde S_k}\right). 
\end{align*}
Similar computation to the proof of Lemma \ref{mainterm} give us that if $n_2/n \to \gamma$
\[ \log\left( \frac{C_{\lfloor nx \rfloor}}{ C_{\lfloor ny \rfloor}}\right) \Rightarrow  \log \left( \frac{\sqrt{ x(\gamma-1+2y)}}{\sqrt{y(\gamma-1+2x)}} \right). \]
And similarly if $n_2\gg n$
\[ \log\left( \frac{C_{\lfloor nx \rfloor}}{ C_{\lfloor ny \rfloor}}\right) \Rightarrow  \log \left( \frac{\sqrt x}{\sqrt y} \right). \]
We now consider the process $y_n$ with increments 
\[\Delta y_{n,k} = \log(S_{k+1}/C_{k+1})-\log (\tilde S_k/\tilde C_k).\]

\noindent In the case where $n_2/n\to \gamma$ we find that 
$$nE(\Delta y_{n,nx})= -\frac{1+a}{2x}-\frac{a}{2(x+\gamma -1)}+o(1),$$
and
$$n\var(\Delta y_{n,nx})= \frac{1}{\beta x}+\frac{1}{\beta(x+\gamma-1)}+o(1).$$
We can also check that  $nE(\Delta y_{n,k})^4=o(1)$.  By Proposition \ref{kurtz} we get the convergence statement in (\ref{process1}).

Similar calculations in the case where $n_2\gg n$ give you that
\[nE(\Delta y_{n,nx})= -\frac{1+a}{2x}+o(1), \hspace{.5cm}
n\var(\Delta y_{n,nx})= \frac{1}{\beta x}+o(1), \hspace{.5cm} nE(\Delta y_{n,k})^4=o(1).\]
This gives us the convergence in (\ref{process2}).

\end{proof}

\subsubsection{Verifying Assumption 2}

We now turn to the tightness conditions.  For (\ref{A2mainterm}) we begin by using the sum bound:
\begin{eqnarray*}
P\left(  \sup_{1\leq k \leq n} \frac{1}{C_k}\cdot \frac{\sqrt{a_n+k}}{\sqrt{a_n+n_2-n+2k}} >M  \right)&\leq & \sum_{k=1}^n P\left( C_k^2 < \frac{a_n+k}{M^2(a_n+n_2-n+2k)}\right)
\end{eqnarray*}
Now recall that for $X\sim \Gamma(p,\theta)$, $Y\sim \Gamma(q,\theta)$ we have that $X/(X+Y)\sim \text{Beta}(p,q)$.  We then start by finding bound on $X$ and $Y$ before turning to the Beta distribution.  In particular we compute $P( X>NE(X))$ and $P(X<E(X)/N)$ for large $N$:  First notice that $E(X)= p\theta$ and so an application of an exponential Chebyshev's inequality with $\theta =1$ gives us
\[P( X>pN)\leq \left(\frac{2}{e^{N/2}}\right)^p, \ \ \ \ \text{ and } \ \ \ \ P(X<p/N) \leq \left( \frac{e}{1+N} \right)^p \]
Using these bounds we can find an upper bound on
\[ P\left(\frac{X}{X+Y}\cdot \frac{p+q}{p} <\frac{1}{N}\right) \]
with exponent $(p\wedge q)$.  

When applied to our original setting, this bound is summable in $k$ (as $n \to \infty$), and moreover the resulting sum may be made as small as desired by increasing $M$.  Therefore, for any $\epsilon>0$ we can choose $M$ such that
\[P\left(  \sup_{1\leq k \leq n} \frac{1}{C_k}\cdot \frac{\sqrt{a_n+k}}{\sqrt{a_n+n_2-n+2k}} >M  \right)<\epsilon,\]
and so the $\kappa_n$ are tight.

For the tightness of the $\kappa_n'$ we use the proof of Lemma 5 in \cite{RR}.  That proof is a reworking of the upper bound in the law of the iterated logarithm and the proof in this situation proceeds with few changes.  We note first that this section will again need to be done in two cases.  We make the following definitions which are analogous to the $A_x^n$ defined in \cite{RRV}.
\begin{equation}
A_{x}^n=\sum_{k=j}^{n-1}  \left[ \log\left(S_{k+1}/C_{k+1}\right)- \log \left( \tilde S_k / \tilde C_k\right) \right] 
- \frac{1+a}{2} \log \left( \frac{j}{i} \right)- \frac{a}{2} \log \left( \frac{j+n\gamma-n}{i+n\gamma-n}\right)
\end{equation}
and
\begin{equation}
B_{x}^n= \sum_{k=j}^{i-1} \left[ \log\left(S_{k+1}/C_{k+1}\right)- \log \left( \tilde S_k / \tilde C_k\right) \right] - \frac{a+1}{2} \log \left( \frac{j}{i} \right).
\end{equation}
for $x \in [x_j,x_{j+1})$.  Here $A_x^n$ will be used in the case where $n_2/n \to \gamma <\infty$ and the associated $T(x)= \frac{1}{\beta}\left( \log \frac{1}{x} + \log \frac{1}{\gamma-1+x}\right)$.  In the case where $n_2\gg n$ the object of interest will be $B_x^n$ with $T(x)= \frac{1}{\beta} \log \frac{1}{x}$.

\medskip

\noindent The changes noted above together with the following claim are sufficient to show the $\kappa_n'$ are tight.  

\begin{remark}
The $T(x)$ as defined in the general theorem differs from the $T(x)$ defined here by a constant.  This omission is permissible because we always consider
\[C(1+T^{3/4}(x)+T^{3/4}(y))\]
and so the difference may be absorbed by the constant term.
\end{remark}

\begin{claim}
\label{moment generating}
For all $\lambda>0$ sufficiently small ($\lambda< (\beta/2)[(a+1)\wedge 1]$ will do), if $n_2/n \to \gamma < \infty$ then
\begin{equation}
E[e^{\lambda A_{x_j}^n}]= \exp \left\{ \frac{\lambda^2}{2\beta}\left[ \log \left( \frac{1}{x_j}\right)+ \log \left( \frac{\gamma}{\gamma-1+x_j}\right)\right]+ \Theta_n(j)\right\}
\end{equation}
With $|\Theta_n(j)|\leq C$ for constant $C=C(a,\beta,\gamma)$.
And if $n_2/n \to \infty$
\begin{equation}
E[e^{\lambda B_{x_j}^n}]= \exp \left\{ \frac{\lambda^2}{2\beta} \log \left(\frac{1}{x_j}\right)+ \Theta_n(j)\right\}
\end{equation}
With $|\Theta_n(j)|\leq C$ for constant $C=C(a,\beta)$. 
\end{claim}

\begin{proof}
[Proof of claim]
We have the following products to consider:  In the case where $n_2/n\to \gamma$ we have
\begin{equation}
\label{mg 1}
E[e^{\lambda A^n_{x_j}}]= E\prod_{k=j}^{n-1} \left( \frac{S_{k+1}}{C_{k+1}} \right)^\lambda \left( \frac{\tilde C_k}{\tilde S_k}\right)^\lambda \left(\frac{k+1}{k}\right)^{\lambda(a+1)/2}\left( \frac{n_2-n+k+1}{n_2-n+k}\right)^{\lambda a/2},
\end{equation}
and in the case where $n_2\gg n$ 
\begin{equation}
\label{mg 2}
E[e^{\lambda B^n_{x_j}}]= E\prod_{k=j}^{n-1} \left( \frac{S_{k+1}}{C_{k+1}} \right)^\lambda \left( \frac{\tilde C_k}{\tilde S_k}\right)^\lambda \left(\frac{k+1}{k}\right)^{2(a+1)/2}
\end{equation}
Let's start by considering the portion of (\ref{mg 1}) and (\ref{mg 2}) of the form
\[P=E\prod_{k=j}^{n-1} \left( \frac{S_{k+1}}{C_{k+1}}\right)^{\lambda}\left(\frac{\tilde C_k}{\tilde S_k}\right)^\lambda.\]
We use independence write this as a product of expectations, then by taking logarithms we can consider each term of the resulting sum separately.  For the $k$-th term this gives us $(\log P)_k=I_k+J_k$ where,
\begin{align*}
I_k=&\ \log \Gamma\left(\frac \beta 2 (n_2-n+k+1)+ \frac \lambda 2\right)-\log \Gamma \left(\frac \beta 2 (n_2-n+k+1)\right)\\
&\hspace{1cm} +\log \Gamma\left(\frac \beta 2 (a_n+n_2-n+k+1)-\frac \lambda 2\right)-\log \Gamma \left( \frac \beta 2 (a_n+n_2-n+k+1) \right)
\end{align*}
and
\begin{align*}
J_k=&\ \log \Gamma\left(\frac \beta 2 (a_n+k+1) -\frac{\lambda}{2}\right)- \log \Gamma \left( \frac \beta 2 (a_n+k+1) \right)\\
&\hspace{2cm} +\log \Gamma\left(\frac \beta 2 k+ \frac{\lambda}{2}\right)- \log \Gamma \left(\frac \beta 2 k\right)
\end{align*}
From the proof of Claim 10 in \cite{RR}, we can conclude that
\[ I_k= \frac{\lambda^2}{2\beta(n_2-n+k)}-\frac{\lambda a}{2}\log\left( 1+ \frac{1}{n_2-n+k}\right)+O(1/k^2)\]
and
\[ J_k = \frac{\lambda^2}{2\beta k}- \frac{\lambda (a+1)}{2} \log\left( 1+\frac{1}{k}\right)+O(1/k^2).\]
The remaining part of (\ref{mg 1}) gives a contribution of 
$$\frac{\lambda a}{2} \log\left(1+ \frac{1}{n_2-n+k}\right)+ \frac{\lambda (a+1)}{2} \log\left( 1+\frac{1}{k}\right).$$
We then use that that $$\sum_{k=1}^n 1/k = \log n+ c+O(1/2n)$$ to establish the claim in the case where $n_2/n\to \gamma$.  In the case where $n_2\gg n$, we have that $I_k=O(1/n_2)$ and so may be folded into the constant term.  Then the remaining term in (\ref{mg 2}) gives a contribution of 
$$ \frac{\lambda (a+1)}{2} \log\left( 1+\frac{1}{k}\right)$$ which establishes the claim in this case.
\end{proof}

\subsubsection{Checking Assumption 3}

To show that the $K_C$ are Hilbert-Schmidt we have the following proposition:

\begin{proposition}
For any constant $C$  and $a>-1$, the integral operators on $L^2[0,1]$ with kernels 
\[k_{C,\gamma}(x,y)=C \exp\left[ CT^{3/4}(x)+CT^{3/4}(y)\right]  \frac{\sqrt{2y+\gamma-1}}{(2x+\gamma-1)^{-1/2}}\frac{(y^2+\gamma y-y)^{a/2}}{(x^2+\gamma x-x)^{(a+1)/2}}\ind_{\{y<x\}}\]
where $T(x)= \frac{1}{\beta}\left(\log \frac{1}{x}+ \log \frac{1}{\gamma-1+x}\right)$
and
$$k_C(x,y)=C \exp\left[ C( \log (1/x))^{3/4}+C(\log(1/y))^{3/4}\right] y^{a/2}x^{-(a+1)/2}\ind_{\{y<x\}}$$
are Hilbert-Schmidt.
\end{proposition}
\begin{proof}
For the operator with kernel $k_{C,\gamma}$ the change of variable $x^2+\gamma x-x =e^{-s}$ and $y^2+\gamma y-y =e^{-t}$ gives the square of the Hilbert-Schmidt norm
$$\int_0^1 \int_0^1 |k_{C,\gamma} (x,y)|^2 dxdy= C^2 \int_0^\infty e^{2Cs^{3/4}+as} \int_s^\infty e^{2Ct^{3/4}-(a+1)t}dtds.$$
Similarly for the operator with kernel $k_C(x,y)$ we can make the change of variables $x=e^{-s}$ and $y=e^{-t}$ to find the same double integral.  This operator is finite if and only if $a>-1$, so both operators are Hilbert-Schmidt.
\end{proof}

\subsection{Checking Assumption 4}

To show that the eigenvalues of $(KK^T)^{-1}$ are simple with $0< \Lambda_0< \Lambda_1< \cdots \uparrow \infty$ we make the following observations: 

\noindent{\bf Observation 1:}  In the case where $n_2/n\to \gamma$ the spectral problem reads  
\begin{align}
f(x)&= \lambda K^TKf(x)= \lambda \int_0^1 k(y,x) \int_0^1 k(y,z) f(z)dzdy\\
f(x)&= \frac{\lambda}{\gamma}\int_x^1 \frac{\sqrt{2x+\gamma-1}}{(x^2+\gamma x-x)^{-a/2}}\frac{2y+\gamma-1}{ (y^2+\gamma y-y)^{a+1}}\exp \left(\int_{x}^{y} \sqrt{ \frac{2s+\gamma-1}{\beta s(s+\gamma-1)}}db_s\right)\notag \\
&\hspace{1.5cm} \times \int_0^y \frac{ \sqrt{2z+\gamma-1}}{(z^2+\gamma z-1)^{-a/2}}\exp\left(\int_{z}^{ y} \sqrt{\frac{2s+\gamma-1}{\beta s(s+\gamma-1)}}db_s\right)f(z)dzdy.\notag
\end{align}
With a bit of rearranging we find that for 
\[g(x)= \frac{(x^2+ \gamma x -x)^{-a/2} }{ \sqrt{2x+\gamma-1}}\exp\left(-\int_{x}^{ 1} \sqrt{\frac{2s+\gamma-1}{\beta s(s+\gamma-1)}}db_s\right)f(x),\]
under the change of variables $(x(x+\gamma-1),y(y+\gamma-1), z(z+\gamma-1))\mapsto (\gamma p,\gamma q,\gamma r)$ with $h(p)=g(x)$ the previous equation reads
\begin{equation}
\label{obs1}
h(p)=\lambda \int_0^1 r^{a}e^{\frac{2}{\sqrt \beta} \hat b(\log 1/r)}h(r)\int_{p\vee r}^1 q^{-(a+1)}e^{-\frac{2}{\sqrt \beta} \hat b(\log 1/q)}dqdr.
\end{equation}

\smallskip
\noindent{\bf Observation 2:}  In the case where $n_2\gg n$ the spectral problem instead reads as
\begin{align*}
f(x)&= \lambda \int_{x}^1 x^{a/2}y^{-a-1} \exp \left(\frac{1}{\sqrt \beta}\int_{x}^y \frac{db_s}{\sqrt{s}}\right) \int_0^y z^{a/2} \exp \left(\frac{1}{\sqrt \beta}\int_{z}^y \frac{db_s}{\sqrt{s}}\right)f(z)dz dy.
\end{align*}
We then take $$g(x)=x^{-a/2} \exp \left(\frac{-1}{\sqrt \beta}\int_{x}^1 \frac{db_s}{\sqrt{s}}\right)f(x)$$
which again yields the equation
\begin{equation}
\label{obs2}
g(x)= \lambda \int_0^1 z^{a} e^{\frac{2}{\sqrt \beta} \hat b(\log 1/z)}\int_{x \vee z}^1 y^{-(a+1)} e^{-\frac{2}{\sqrt \beta} \hat b(\log 1/y)}g(z)dydz.
\end{equation}
Equations (\ref{obs1}) and (\ref{obs2}) are equivalent to the eigenvalue equation $\psi = \lambda \mathfrak{G}_{\beta,a} \psi$ (\cite{RR}, Proof of Theorem1).  The simplicity of the eigenvalues of $\mathfrak{G}_{\beta,a}$ and their ordering with $0<\Lambda_0(\beta,a)<\Lambda_{1}(\beta,a)<\cdots \uparrow \infty$ are discussed by Ram\'irez and Rider in \cite{RR}.

\smallskip

\noindent We have showed that Theorem \ref{inverselimit} applies to $W_{n,\beta}$ and the spectrum of the limiting operator is the same as that of $\mathfrak{G}_{\beta,a}$.  This completes the proof of Theorem \ref{HEL}.


\section{Appendix: Proof of Theorem \ref{tridiagonal}}

We will be working primarily with a lower bidiagonal matrix $M$, and the symmetric tridiagonal matrix $MM^T$.  For convenience we will adopt the following notations.  Denote the diagonal entries of our symmetric tridiagonal matrix $MM^T$ by $\Ba = \{a_1,...,a_n\}$ and its off-diagonal entries by $\Bb=\{b_1,..., b_{n-1} \}$. For the entries of the bi-diagonal matrix $M$ use $\Bx = \{x_1,...,x_n\}$ to denote the diagonal entries and  $\By=\{y_1,...,y_{n-1} \}$ for its sub-diagonal entries.  We will also use another tridiagonal matrix which arises in the following lemma (see e.g. \cite{DF}):

\begin{lemma}
\label{doubling}

Let $M$ be an $n \times n$ bidiagonal matrix with $a_1, a_2, \dots, a_n$ in the diagonal and $b_1, b_2, \dots, b_{n-1}$ in the off-diagonal. Consider the $2n\times 2n$ symmetric tridiagonal matrix 
$L$ which has zeros in the main diagonal and $a_1, b_1, a_2, b_2, \dots, a_n$ above and below the diagonal. If the singular values of $M$ are $\lambda_1, \lambda_2, \dots, \lambda_n$ then the eigenvalues of $L$ are $\pm \lambda_i, i=1\dots n$.
\end{lemma}

Recall the definitions of $M_{\beta, n}, C_k $ and $\tilde C_k$ from the statement of Theorem \ref{tridiagonal}.  It will be convenient write $\{C_i,S_i\}_{i=1}^n$ and $\{\tilde C_i, \tilde S_i\}_{i=1}^{n-1}$ in terms of a set of random angles $\Balpha = \{\alpha_1, ..., \alpha_n \}$ and $\Btheta=\{\theta_1,..., \theta_{n-1}\}$ where
$C_k = \cos(\alpha_k)$ and $\tilde C_k = \cos(\theta_k)$. Then, $S_k = \sin(\alpha_k)$ and $\tilde S_k = \sin(\theta_k)$, and the densities of the random angles $\alpha_k$ and $\theta_k$ are:
\begin{eqnarray*}
 f^{\alpha}_k & = &\frac{ 2}{B\left( \frac{\beta}{2}(n_1-n+k), \frac{\beta}{2}(n_2-n+k)\right)} \cos^{2a + \beta(k-1)+1}(\alpha_k) \sin^{2b + \beta(k-1)+1}(\alpha_k)\\
 f^{\theta}_k & = & \frac{ 2}{B\left( \frac{\beta}{2}k, \frac{\beta}{2}(n_1+n_2-2n+k+1)\right)} \cos^{\beta k}(\theta_k) \sin^{2a + 2b + \beta(k-1) + 3}(\theta_k)
\end{eqnarray*}
with $a,b$ defined as in Theorem \ref{tridiagonal} and $B(x,y)$ the beta function.  By independence, the joint density of $(\Balpha, \Btheta)$ is given by the product of the densities.  For convenience we will denote 
\[\tilde Z_{\beta,n}= \frac{2^{2n-1}}{\prod_{k=1}^nB\left( \frac{\beta}{2}(n_1-n+k), \frac{\beta}{2}(n_2-n+k)\right) \prod_{k=1}^{n-1}B\left( \frac{\beta}{2}k, \frac{\beta}{2}(n_1+n_2-2n+k+1)\right)}\]
This will be the necessary normalizing constant.

We  will now map $(\Balpha, \Btheta)$ to the entries of $M$ and from there to $(\Blambda, \Bq)$ where the $\lambda_i$ are the eigenvalues of $MM^T$ and the  $q_i$ are the positive leading entries of the corresponding normalized eigenvectors.  This second map will actually by the composition of several maps.  

\noindent From the angles we map to the entries of the bidiagonal matrix:
\begin{lemma}The Jacobian of the transform $T:(\Balpha, \Btheta) \to (\Bx, \By)$, where
\begin{eqnarray*}
 x_k & = &  \cos(\alpha_k) \sin(\theta_k),\quad  k=1,\cdots,n \\
 y_k & = & \sin(\alpha_{k+1}) \cos(\theta_k),\quad k=1,\cdots,n-1.
\end{eqnarray*}
(for notational convenience we take $\theta_n=\pi/2$) is given by
 \begin{eqnarray*}
   J_{T} = \frac{\sin^2(\alpha_n)}{\sin(\alpha_1)} \prod_{k=1}^{n-1} \sin^2(\alpha_k)\sin^2(\theta_k).
 \end{eqnarray*}
\end{lemma}

\begin{proof}
The matrix of the partial derivatives with ordering $(\alpha_1,...,\alpha_n,\theta_1,...,\theta_{n-1}) \to (x_1,...,x_n, y_1,...,y_{n-1})$ has diagonal entries
\[- \sin \alpha_1 \sin \theta_1,....,-\sin \alpha_{n-1}\sin \theta_{n-1}, -\sin \alpha_n, - \sin \alpha_2 \sin \theta_1,..., -\sin \alpha_n \sin \theta_{n-1}.\]
Row and column reduction gives us that the determinant is given by the product of the diagonal, therefore
\[J_T= -\frac{\sin^2(\alpha_n)}{\sin(\alpha_1)} \prod_{k=1}^{n-1} \sin^2(\alpha_k)\sin^2(\theta_k)
\qedhere
\]
\end{proof}

\noindent We finish by mapping from the bidiagonal to the tridiagonal matrix, which in turn is mapped to by the eigenvalues.  Denote by $q_1,...,q_n$ the leading entries of the eigenvectors associated with the ordered eigenvalues $\lambda_1<\lambda_2<\cdots< \lambda_n$ of the tridiagonal matrix normalized so that $q_i>0$ and $\sum q_i^2 =1$.  

\begin{lemma}[\cite{DE}, Lemmas 2.7, 2.9 and 2.11] 
For $\Bx,\By,\Ba,\Bb,\Blambda$ and $\Bq$ defined as above we have the following:

\begin{enumerate}
\item The Jacobian of the map $\psi: (\Bx, \By) \to (\Ba, \Bb)$ can be written as
\[ J_\psi = 2^n x_1 \prod_{i=2}^{n} x_i^2.\]

\item The Vandermonde determinant for the ordered eigvenvalues of a symmetric dridiagonal matrix with positive sub-diagonal $b=(b_{n-1},...,b_1)$ is given by
\[ \Delta ( \Blambda)= \prod_{i<j}(\lambda_i-\lambda_j)= \frac{\prod_{i=1}^{n-1} b_i^i}{\prod_{i=1}^n q_i}.\]

\item The Jacobian of the map $\phi:(\Ba, \Bb) \to (\Blambda, \Bq)$ can be written as
\[ J_\phi = \frac{\prod_{i=1}^{n-1} b_i}{\prod_{i=1}^n q_i}.\]

\end{enumerate}

\end{lemma} 

\noindent Written in terms of our $\Bx, \By, \Ba,$ and $\Bb$ with $Q_n= \prod_{i=1}^n q_i$ we get
 \begin{align*}
   J_{\phi} & =  \frac{1}{Q_n} \sin(\alpha_n) \cos (\alpha_n) \cos (\theta_{n-1}) \prod_{k=1}^{n-2}\sin(\alpha_{k+1})\cos(\alpha_{k+1})\sin(\theta_{k+1})\cos(\theta_k) \\
   J_{\psi} & =  2^n \cos(\alpha_1)\sin(\theta_1)\cos^2(\alpha_n) \prod_{k=2}^{n-1}\cos^2(\alpha_k) \sin^2(\theta_k).
 \end{align*}

\noindent The remainder of the proof of Theorem \ref{tridiagonal}, is to show that the Jacobian of the transformation gives the desired result.
\begin{align*}
 d(\Blambda, \Bq) &=   \frac{J_{\phi}}{J_{\psi} \times J_T} \, d(\Balpha, \Btheta)\\
 &= \tilde Z_{\beta,n} \frac{2^{-n}}{ Q^{1-\beta}_n} \left( \frac{1}{Q_n} \prod_{k=1}^n \cos^{k-1}\alpha_k \sin^{k-1}\alpha_k \prod_{k-1}^{n-1} \cos^k \theta_k \sin^{k-1}\theta_k\right)^\beta\\
		& \hspace{2cm} \times  \quad \prod_{k=1}^n \cos^{2a}\alpha_k \sin^{2b} \alpha_k \times \prod_{k=1}^{n-1} \sin^{2a+2b}\theta_k. 
\end{align*}

\noindent We substitute in the following pieces:  first, notice that

\[(\det  M_{n,\beta})^2 = \prod_{i=1}^n \lambda_i = \prod_{k=1}^n \cos^2 \alpha_k \cdot \prod_{k=1}^{n-1} \sin^2 \theta_k.\]

\noindent Then, applying Lemma \ref{doubling}, the singular values of $M_{n,\beta}$ are the squares of the eigenvalues of the symmetric tridiagonal matrix $L$ with zeroes in the main diagonal and 
\[
C_n,\,  S_n \tilde C_{n-1},\,  C_{n-1} \tilde S_{n-1},\,  \dots,\, S_2 \tilde C_1,\, C_1 \tilde S_1,
\]
in the off-diagonal. We can check that $L+I$ can be written as $AA^{T}$ where $A$ is the  bidiagonal matrix with 
\[
1, S_n, \tilde S_{n-1}, S_{n-1}, \dots, S_1, \ \ \ \text{ and } \ \ \ 
C_n, \tilde C_{n-1}, C_{n-1}, \tilde C_{n-2}, \dots, C_1
\]
in the diagonal and below the diagonal respectively.  Using the characterization from Lemma \ref{doubling}, we find that 
\[
\det (L+I) =   \prod_{{}^{k=1,\cdots,n}_{\,\, \epsilon=+,-}} 1 + \epsilon\sqrt{\lambda_k}  = \prod^n_{k=1} (1-\lambda_k) = (\det A)^2 = \prod_{k=1}^n \sin^2 \alpha_k \prod_{k=1}^{n-1} \sin^2 \theta_k.
\]

\begin{remark}
At this point one could again apply Lemma \ref{doubling} to $A$ to find a $4n \times 4n$ matrix with zeros in the diagonal and
\[ 1, C_n , S_n, \tilde C_{n-1}, \tilde S_{n-1}, C_{n-1},...,C_1,S_1\]
above and below the diagonal.  This matrix will have eigenvalue $\pm \sqrt{1\pm \sqrt{\lambda_k}}$.
\end{remark}

Lastly recall (\cite{DE}, Lemma $2.7$) that we have 
\[\Delta(\lambda)= \frac{1}{Q_n} \prod_{k=1}^{n-1} b_k^k= \frac{1}{Q_n}\prod_{k=1}^n \sin^{k-1}\alpha_k \cos^{k-1}\alpha_k \cdot \prod_{k=1}^{n-1} \cos^k\theta_k \cdot \prod_{k=1}^{n-1} \sin^{k-1}\theta_k.\]
Making all the appropriate substitutions this gives us that
\begin{eqnarray*}
\frac{J_{\phi}}{J_{\psi} \times J_T} \, d(\Balpha, \Btheta) & = & \tilde Z_{\beta,n}  \frac{2^{-n}}{ Q^{1-\beta}_n} \prod_{i<j}|\lambda_i- \lambda_j|^\beta \cdot  \prod_{k=1}^n \lambda_k^a \cdot \prod_{k=1}^n (1-\lambda_k)^b,
\end{eqnarray*}
From this we can see that the joint density function of $\Bq$ and $\Blambda$ separate.  As in the Hermite and Laguerre cases we have that $\Bq \sim ( \chi_\beta,...,\chi_\beta)$ normalized to unit length \cite{DE}.  This give us that for the unordered eigenvalues
\[f_{\beta,n,n_1,n_2}(\Blambda)= \frac{\tilde Z_{\beta,n}}{2^n n!}\frac{\left[\Gamma(\frac{\beta}{2})\right]^n}{\Gamma\left( \frac{\beta n}{2}\right)}\prod_{i<j}|\lambda_i- \lambda_j|^\beta \cdot  \prod_{k=1}^n \lambda_k^a \cdot \prod_{k=1}^n (1-\lambda_k)^b.\]

\noindent This completes the proof of Theorem \ref{tridiagonal}.
\hfill $\square$

\begin{remark}
The normalizing constant on the final density can be written as
\[ Z_{\beta,n} = \left[ \Gamma\left(\frac{\beta}{2}\right) \right]^n \prod_{k=1}^n \frac{\Gamma \left( \frac{\beta}{2} (n_1+n_2-n+k)\right)}{\Gamma\left( \frac{\beta}{2}(n_1-n+k)\right)\Gamma\left( \frac{\beta}{2}(n_2-n+k)\right)\Gamma\left( \frac{\beta}{2}k\right)}.\]
This gives an alternate derivation for the Selberg integral.
\end{remark}


\end{document}